\documentclass[11pt]{amsart}

\usepackage{amssymb}
\usepackage[pdftex]{graphicx}
\usepackage{epstopdf}
\usepackage{float}

\usepackage[margin=1.28in]{geometry}

\usepackage{latexsym}
\usepackage[shortlabels]{enumitem}
\usepackage{hyperref}
\usepackage{tikz-cd}

\usepackage{chngcntr}
\usepackage{etoolbox}
\newtheorem{theorem}{Theorem}[subsection]
\newtheorem{lemma}[theorem]{Lemma}
\newtheorem{proposition}[theorem]{Proposition}

\theoremstyle{definition}
\newtheorem{definition}[theorem]{Definition}
\newtheorem*{acknowledgments}{Acknowledgments}

\theoremstyle{remark}
\newtheorem{remark}[theorem]{Remark}
\newtheorem{Example}[theorem]{Example}

\makeatletter
\g@addto@macro\@floatboxreset\centering
\makeatother

\DeclareMathOperator{\Spec}{\underline{Spec}}
\DeclareMathOperator{\Pic}{Pic}

\newcommand{\numset}[1]{\mathbb{#1}}

\newcommand{\Z}{\numset{Z}}
\newcommand{\Q}{\numset{Q}}

\newcommand{\C}{\mathbb{C}}
\newcommand{\F}{\mathbb{F}}
\newcommand{\X}{\mathcal{X}}

\newcommand{\PP}{\mathbb{P}}
\newcommand{\OO}{\mathcal{O}}
\newcommand{\D}{\mathcal{D}}
\newcommand{\ZZ}{\mathcal{Z}}
\newcommand{\mf}{\mathfrak{U}}
\newcommand{\mfp}{\mathfrak{P}}

\newcommand{\surj}{\twoheadrightarrow}

\title{Genus six curves, K3 surfaces, and stable pairs}

\author[J.~ Ross~Goluboff]{J. Ross Goluboff}
\address{Department of Mathematics, Boston College, Chestnut Hill, MA 02467, USA}

\email{\href{mailto: goluboff@bc.edu}{goluboff@bc.edu}}

\begin{document}

\begin{abstract}
A general smooth curve of genus six lies on a quintic del Pezzo surface. In \cite{AK11}, Artebani and Kond\=o construct a birational period map for genus six curves by taking ramified double covers of del Pezzo surfaces. The map is not defined for special genus six curves. In this paper, we construct a smooth Deligne-Mumford stack $\mfp_0$ parametrizing certain stable surface-curve pairs which essentially resolves this map. Moreover, we give an explicit description of pairs in $\mfp_0$ containing special curves. 
\end{abstract}

\maketitle

\section{Introduction}
\numberwithin{theorem}{section}
\counterwithin{theorem}{section}
\numberwithin{equation}{section}
\counterwithin{equation}{section}

In \cite{AK11}, the authors construct a birational period map 
\begin{equation*}
\varphi: \mathcal{M}_6 \dashrightarrow D/\Gamma,
\end{equation*}
where the source denotes the moduli space of genus six curves and the target parametrizes certain lattice-polarized $K3$ surfaces (see, for example, \cite[Section 1]{Dol96}). Their construction of $\varphi$ is as follows. The canonical model of a general smooth curve $C$ of genus six is a quadric section of a unique smooth quintic del Pezzo surface $\Sigma_5$ embedded anti-canonically in $\PP^5$. The double cover of $\Sigma_5$ branched along $C$ will be a $K3$ surface. Taking the period point of this surface defines $\varphi$. More precisely, the output of $\varphi$ is a lattice-polarized $K3$ surface where the lattice has rank $5$ (note that $\mathcal{M}_6$ and $D/\Gamma$ are $15$-dimensional, while the moduli space of polarized $K3$ surfaces is $19$ dimensional).

A smooth curve of genus six is called \emph{special} if it is one of the following four types: hyperelliptic, trigonal, bielliptic, or plane quintic. The canonical model of any non-special smooth curve of genus six lies on a unique weak del Pezzo surface (see, for example, {\cite[Proposition 1.1]{AK11}}), so $\varphi$ extends over such curves. Note that $\varphi$ does not extend over special curves; the canonical models of these curves do not lie on weak quintic del Pezzo surfaces in $\PP^5$. Moreover, Artebani and Kond\=o prove that the birational period map $\varphi$ induces an isomorphism
\begin{equation*}
\mathcal{M}_6 \setminus \{\text{special curves}\} \rightarrow (D \setminus \mathcal{H})/\Gamma,
\end{equation*}
where $\mathcal{H}$ denotes a \emph{discriminant divisor}. Artebani and Kond\=o show that $\mathcal{H}$ has $3$ irreducible components and that the general member of these components corresponds to a genus six curve with a node in $\Sigma_5$, the union of a plane quintic and a line in $\PP^2$, and the union of a trigonal curve $C$ of genus six and a section $e \in |K_C-2g^1_3|$ in $\PP^1 \times \PP^1$ respectively (\cite[Theorem 0.2]{AK11}). The $K3$ surfaces corresponding to such curves are also constructed via double covers branched along these curves.

The goal of this paper is to construct a space resolving the indeterminacy of $\varphi$ and give a modular interpretation for this space. Studying birational period maps has been a topic of significant interest in the literature. Shah in \cite{Sh80} defines a period map for the GIT (geometric invariant theory) space of plane sextics by taking ramified double covers of $\PP^2$. The indeterminacy occurs precisely along the locus of triple conic curves, which he resolves by blowing it up. Kond\=o in \cite{Kon00} defines a birational period map for curves of genus three by taking four-fold cyclic covers of $\PP^2$ branched along quartic curves, which induces an isomorphism between the moduli space of non-hyperelliptic curves of genus three and the arithmetic quotient of a period domain minus a discriminant divisor. Similarly, Kond\=o in \cite{Kon02} constructs a birational period map for genus four curves by taking triple covers of quadric surfaces in $\PP^3$ branched along non-hyperelliptic curves. Artebani in \cite{Ar09} expands upon Kond\=o's work in genus three by considering the GIT space for plane quartics and completely resolves the indeterminacy of the period map on the level of compactifications by blowing up the double conic locus. In \cite{CMJL12}, the authors expand upon Kond\=o's work in genus four by constructing a GIT model for $\overline{\mathcal{M}}_4$ and resolving the period map using techniques from Looijenga. In \cite{LO16}, Laza and O'Grady study the relationship between GIT and Satake-Baily-Borel compactifications of quartic $K3$ surfaces.

To resolve the period map $\varphi$ for genus six curves, rather than using GIT, we appeal to Hacking's theory of stable pairs developed in \cite{Hac01}, \cite{Hac04} and generalized in \cite{DH18}. A stable pair is a surface-curve pair satisfying certain properties for moduli theoretic purposes. The moduli spaces of stable pairs constructed in these papers are modified versions of the KSBA (Koll\'ar, Shepherd-Barron, Alexeev) compactification. In Section 3, we will formally define stable pairs and their allowable ($\mathbb{Q}$-Gorenstein) families. We remark that work of Hyeon and Lee in \cite{HL10} reveals that Hacking's compactification of plane quartic curves is already useful for the analogous period map question in genus three: in this case, Hacking's space extends the period map over hyperelliptic curves. We should also note that in \cite{AET19}, the authors construct a stable pair compactification of $K3$ surfaces of degree $2$.

Using Hacking's framework, we can consider a moduli stack $\mathfrak{P}$ of stable pairs whose general point is a pair of the form $(\Sigma_5, C)$, where $C$ is smooth and of class $-2K_{\Sigma_5}$ (see Definition~\ref{aux_stacks}). We define two open substacks of $\mfp$ that will be necessary for us:
\begin{definition} \label{stacks}
Let $\mfp_0 \subset \mfp$ be the open substack parameterizing stable pairs $(X, D)$ such that: 

\begin{enumerate}
\item The surface $X$ has only combinations of du Val, index two cyclic quotient singularities, and simple elliptic singularities.

\item The curve $D$ has at worst ADE singularites and avoids the singularities of $X$. 
\end{enumerate}
In $(1)$, we allow ``empty" combinations of singularities, hence the surface $X$ may have only some of the listed singularities or may even be smooth. 

Let $\mfp_0^\text{sm} \subset \mfp_0$ be the open substack parametrizing stable pairs $(X, D)$ satisfying properties $(1)$ and $(2)$ with $D$ smooth.
\end{definition}

We remark that an index $2$ cyclic quotient singularity arising on a stable pair in $\mathfrak{P}$ is a \emph{class $T$ singularity} (see Definition~\ref{classT}).

The main result of this paper is the following:

\begin{theorem} \label{mainthm}

The stack $\mfp_0$ is Deligne-Mumford, smooth, and fits into the diagram 

\begin{center}
\begin{tikzcd}[column sep=small]
& \mathfrak{P}_0 \arrow["j", dl, dashed, swap]  \arrow[dr, "\tilde \varphi"] & \\
\overline {\mathcal{M}}_6 \arrow[dashed]{rr}{\varphi} & & (D/\Gamma)^\ast
\end{tikzcd}
\end{center}
where $j$ is the natural (birational) forgetting map and $\tilde \varphi$ extends the double cover construction of $\varphi$. Moreover, the map $j$ restricts to a surjective morphism 
\begin{equation*}
j|_{\mfp_0^\text{sm}}: \mfp_0^\text{sm} \surj \mathcal{M}_6 \setminus \mathcal{H}_6,
\end{equation*}
where $\mathcal{H}_6$ denotes the hyperelliptic locus.

\end{theorem}

In the statement of this theorem, $(D/\Gamma)^\ast$ denotes the Satake-Baily-Borel compactification of $D/\Gamma$. The content of this theorem is that $\mfp_0^\text{sm}$ resolves the map $\varphi$ over plane quintic, trigonal, and bielliptic curves (all special curves except the hyperelliptics). The proof of this theorem will involve explicit construction of stable pairs $(X, D)$ containing special genus six curves. Using these pairs, we verify surjectivity of $j$ over smooth non-hyperelliptic curves. Table~\ref{table:pairs} in Section 4 gives a complete list of the pairs we construct. We note that these pairs lie in three distinct boundary loci in $\mfp_0$, denoted $\ZZ_1$, $\ZZ_2$, and $\ZZ_3$. By ``boundary" here, we mean pairs $(X, D)$ such that $X$ is singular and does not have du Val singularities. We describe these boundary loci below by giving the dimension and general member of each.

\begin{enumerate} \label{bdry}
\item $\ZZ_1$: $14$ dimensional (a divisor). The general member is a pair $(X, D)$ where $X$ is constructed by choosing a line transverse to a smooth plane quintic curve in $\PP^2$, blowing up the $5$ points of intersection, and contracting the strict transform of the line. This contraction produces a $\frac{1}{4}(1,1)$ cyclic quotient singularity. The curve $D$ is the image of the quintic in $X$.

\item $\ZZ_2$: $14$ dimensional (a divisor). The general member is a pair $(X, D)$ where $X$ is constructed by first choosing a trigonal curve of genus six $C$ on $\PP^1 \times \PP^1$ and a ruling $e$ meeting $C$ transversely in $4$ points. Then we blow up the four points of intersection between $C$ and $e$ and contract the strict transform of $e$. This contraction also produces a $\frac{1}{4}(1,1)$ cyclic quotient singularity. The curve $D$ is the image of $C$ in $X$. 

\item $\ZZ_3$: $10$ dimensional. The general member is a pair $(X, D)$ where $X$ is a cone in $\PP^5$ over an elliptic curve embedded in $\PP^4$ via a degree $5$ line bundle and $D$ is a smooth quadric section of $X$ (a bielliptic curve).
\end{enumerate}

Moreover, we verify that given any pair $(X,D)$ in $\mfp_0$, the double cover of $X$ branched along $D$ yields a (degeneration of a) $K3$ surface with
``insignificant limit singularities" (see \cite{Sh79}, \cite{Sh80} for the definition of such singularities). In dimension $2$, these are precisely the Gorenstein semi-log canonical (slc) singularities. Since the period map for $K3$ surfaces extends over degenerations with such singularities, the map $\tilde \varphi$ is indeed a morphism as asserted in the theorem (see Proposition~\ref{K3} and Proposition~\ref{period}). The $K3$ surfaces associated to the pairs $(X, D)$ over plane quintics and trigonal curves will be closely related to components of the discriminant divisor described by Artebani and Kond\=o (we discuss this in Remark~\ref{notable}). 

We also remark that we can resolve the map $j$ via the simultaneous stable reduction of Casalaina-Martin and Laza for families of ADE curves (see {\cite[Theorem 3.5, Corollary 6.3]{CML13}}). We will see in Section 4 that this process yields a space that also resolves $\varphi$ over the hyperelliptic curves (see Remark~\ref{sim_st}).

We remark that one motivation for using stable pairs to resolve $\varphi$ stems from the Hassett-Keel program for genus six curves. In \cite{Mu14}, M\"uller shows that the final log canonical model of $\overline {\mathcal{M}}_6$ parametrizes quadric sections of $\Sigma_5$. Given a one-parameter degeneration of quadric sections of $\Sigma_5$ over the germ of a smooth curve, we can modify it so that the new special fiber is a stable pair. This stable reduction process involves applying techniques from the minimal model program. By enumerating singular quadric sections of $\Sigma_5$ by topological type and running this process, the hope is to construct a complete list of the pairs in $\mfp$. We describe some explicit examples of this process in Section $5$, but since the list of singular quadric sections of $\Sigma_5$ is quite long, we save this approach for future work. 

We should also remark that in \cite{HL10}, the authors identify both Hacking's compactification of plane quartic curves and the Satake-Baily-Borel compactification of Kond\=o's period space in \cite{Kon00} with certain log canonical models of $\overline{\mathcal{M}}_3$. One might ask: How do $\mfp$ and $(D/\Gamma)^\ast$ fit into the Hassett-Keel story for genus six curves? We also leave this question for future work.

The paper is organized as follows. Section $2$ will describe some salient features of the geometry of special genus six curves. In Section $3$, we will recall the theory of stable pairs and establish a smoothability criterion for such pairs with mild surface singularities. Section $4$ will be devoted to proving Theorem~\ref{mainthm}. The proof will entail explicitly constructing surface-curve pairs using the geometry of special curves and then applying the smoothability criterion. Section $5$ gives some examples of computing stable limits of one-parameter degenerations of quadric sections of $\Sigma_5$, and we recover some of the pairs constructed in Section $4$.

\section{Geometry of special curves}
\numberwithin{theorem}{section}
\counterwithin{theorem}{subsection}
\numberwithin{equation}{section}
\counterwithin{equation}{subsection}
In this section, for each smooth non-hyperelliptic special curve $C$ of genus six mentioned in the introduction, we give a natural surface $S$ into which $C$ embeds. This will guide our search for stable pairs containing a given curve. We also introduce stratifications of plane quintic and trigonal curves after specifying certain marked divisors. Throughout this paper, $\F_{n}$ will denote the Hirzebruch surface $\PP(\OO_{\PP^1} \oplus \OO_{\PP^1}(-n))$.

\vspace{.1in}
\subsection{Plane quintics} Of course, such a curve embeds in $\PP^2$. 

\begin{definition}
A \emph{marked plane quintic curve} is a pair $(C, E)$ where $C$ is a plane quintic curve and $E$ is a hyperplane section. 
\end{definition}

In Section $4$, for each marked smooth plane quintic curve $(C, E)$, we exhibit a stable pair containing $C$. Marked smooth plane quintic curves $(C, E)$ are stratified by partitions $(a_1, \dots ,a_5)$ of $5$; the partition represents the non-zero coefficients of the points in the support of $E$. For example, a pair $(C, E)$ of type $(1,1,1,1,1)$ means that $E=\ell|_C$, where $\ell$ is a line transverse to $C$. On the other hand, a pair $(C, E)$ of type $(5)$ means that $E=\ell|_C$, where $\ell$ meets $C$ in a single point with intersection multiplicity $5$.

\vspace{.1in}

\subsection{Trigonal curves} Recall the construction of a rational normal surface scroll in $\PP^{g-1}$. For two non-negative integers $a$ and $b$ such that $a+b=g-2$, a rational normal surface scroll $S_{a,b}$ is the join of two rational normal curves of degrees $a$ and $b$ with complementary linear spans. Equivalently, $S_{a,b}$ can be defined as the rational ruled surface $\PP(\OO_{\PP^1}(-a)\oplus \OO_{\PP^1}(-b))$. 

Now, consider a smooth trigonal curve $C \subset \PP^{g-1}$. The linear system of quadrics containing $C$ cuts out a rational normal surface scroll $S_{a,b}$ (see \cite[Proposition 3.1]{ACGH85}). We now define some numerical invariants associated to the embeddings of smooth trigonal curves in scrolls that help us stratify such curves. 

\begin{definition}
Let $S_{a,b}$ denote the rational normal surface scroll containing a given smooth trigonal curve $C$. The quantity $M=|a-b|$ is called the \emph{Maroni invariant} of $C$.
\end{definition}    

Tautologically, a smooth trigonal curve $C$ of Maroni invariant $M$ embeds into the Hirzebruch surface $\F_M$. We note that for genus six, there are only two possible values for $M$: $0$ and $2$. When $M=0$, by genus considerations, such a curve has class $3e+4f$ on $\PP^1 \times \PP^1$, where $e$ and $f$ denote the classes of the two rulings.

When $M=2$, such a curve has class $3e+7f$ on $\F_2$, where $e$ denotes the negative section and $f$ denotes the fiber class of the projection $\F_2 \rightarrow \PP^1$ (the latter cuts out the $g^1_3$ on $C$). The negative section has a unique point of intersection with $C$; denote this point $p$. Let $f_p$ denote the unique fiber containing $p$.

Any smooth trigonal curve $C$ of genus six has not only a unique $g_3^1$ but also a unique $g^1_4$ of class $K_C-2g^1_3$. If $C$ has Maroni invariant $0$, then this $g_4^1$ is cut out by $e$ on $\PP^1 \times \PP^1$. If $C$ has Maroni invariant $2$, the $g_4^1$ is cut out by $e+f$. 

\begin{definition}
A \emph{marked trigonal curve of genus six} is a pair $(C, E)$ where $C$ is a trigonal curve of genus six and $E$ is a divisor in the unique $g^1_4$ associated to $C$. 
\end{definition}

In Section $4$, for each marked smooth trigonal curve of genus six $(C, E)$, we exhibit a stable pair containing $C$. We will use the following notation to stratify marked smooth trigonal curves of genus six $(C, E)$:

\begin{enumerate}
\item \emph{Type $(2; [a_1], a_2, a_3, a_4)$}: A pair $(C, E)$ is of this type if $C$ has Maroni invariant $2$ and
\begin{equation*}
E=a_1p+\displaystyle \sum_i a_i p_i.
\end{equation*}
Note that when $C$ has Maroni invariant $2$, the point $p$ is always in the support of $E$. The $a_i$ necessarily form a partition of $4$. Note that $a_1>1$ if and only if $E=(e+f_p)|_C$, and $a_1=1$ if and only if $E=(e+f_0)|_C$ for some fiber $f_0 \neq f_p$. 

\item \emph{Type $(0; b_1, b_2, b_3, b_4)$}: A pair $(C, E)$ is of this type if $C$ has Maroni invariant $0$ and
\begin{equation*}
E=\displaystyle \sum_i b_ip_i.
\end{equation*}
The $b_i$ necessarily form a partition of $4$. 

\end{enumerate}

\subsection{Bielliptic curves}\label{biell_geo} A bielliptic curve is one that admits a $2:1$ cover of an elliptic curve. A smooth genus six bielliptic curve can be realized as a quadric section of a cone in $\PP^5$ over a smooth elliptic curve embedded in $\PP^4$ via a degree $5$ line bundle. The curve avoids the vertex of the cone (see %{\cite[]{ACGH85}}, \cite{Co12}, 
\cite[Lemma 3.3]{Ko05}, for example). Moreover, for a bielliptic curve of genus six (in fact for genus greater than five), the bielliptic involution is unique (see \cite[Chapter $5$]{Acc94}, for example) and the quotient by this involution is isomorphic to the exceptional elliptic curve for the minimal resolution of this cone.

\section{Moduli of stable pairs}
\numberwithin{theorem}{section}
\counterwithin{theorem}{section}
\numberwithin{equation}{section}
\counterwithin{equation}{section}
In this section, we outline the theory of stable pairs. We refer the reader to \cite{Hac01}, \cite{Hac04}, and \cite{DH18} for more details. The key idea is that the forthcoming definitions allow us to construct the moduli stacks $\mfp$, $\mfp_0$, and $\mfp_0^\text{sm}$ (recall Definition~\ref{stacks}).
\begin{definition}
Let $X$ be a surface and $D$ an effective $\Q$-divisor on $X$. The pair $(X,D)$ is said to be \emph{semi log canonical (slc)} (resp. semi log terminal (slt)) if the following conditions hold:

\begin{enumerate}
\item $X$ is Cohen-Macaulay and has at worst normal crossings singularities in codimension $1$.
\item The divisor $K_X +D$ is $\Q$-Cartier. 
\item Let $\nu: X^\nu \rightarrow X$ denote the normalization of $X$, $\delta$ the double curve of $X$, $D^\nu$ and $\delta^\nu$ the inverse images of $D$ and $\delta$. Then the pair $(X^\nu, \delta^\nu+D^\nu)$ is log canonical (resp. log terminal). 
\end{enumerate}
\end{definition}

\begin{definition}[{\cite[Definition 2.1]{DH18}}]\label{stable_pair}
Let $m,n$ be positive co-prime integers with $m < n$. Let $X$ be a projective, reduced, connected, Cohen-Macaulay surface and $D$ an effective Weil divisor on $X$. We say that $(X,D)$ is a \emph{stable pair of type (m,n)} if the following conditions hold:

\begin{enumerate}
\item No component of $D$ is contained in the singular locus of $X$. 
\item For some $\epsilon>0$, the pair $(X, (m/n+\epsilon)D)$ is slc, and the divisor $K_X+(m/n+\epsilon)D$ is ample.

\item The divisor $nK_X+mD$ is linearly equivalent to zero.

\item $\chi(\OO_X)=1$. 
\end{enumerate}
\end{definition}

\begin{definition}[{\cite[Definition 2.3]{DH18}}] \label{qg1}
A \emph{$\Q$-Gorenstein family of stable pairs of type $(m,n)$} is a pair $(\pi: \X \rightarrow T, \D \subset \X)$, where $\D$ is a relative effective Weil divisor and $\pi$ is a flat, proper, Cohen-Macaulay morphism with slc surfaces as geometric fibers, satisfying the following additional conditions:

\begin{enumerate}
\item $\omega_\pi^{[i]}$ commutes with base change for every $i \in \Z$, and on each geometric fiber, some reflexive power of $\omega_\pi$ is invertible.
\item $\OO_X(\D)^{[i]}$ commutes with base change for every $i \in \Z$.
\item Each geometric fiber is a stable pair of type $(m,n)$.

\end{enumerate}
\end{definition}

For brevity, we will occasionally write ``stable pair" and omit ``of type $(m,n)$." We will eventually specialize to the case $(m,n)=(1,2)$. Geometrically, $\Q$-Gorenstein families of stable pairs are those which lift locally to canonical coverings (to be defined below). It is often more convenient to use this geometric definition when discussing $\Q$-Gorenstein deformations of singularities. We formally define canonical cover and the geometric version of $\Q$-Gorenstein deformation of a stable pair below, following \cite{Hac04}. Recall that the \emph{index} of a $\mathbb{Q}$-Cartier Weil divisor $D$ at a point $P$ in a normal variety $X$ is the smallest positive integer such that $ND$ is Cartier near $P$. 

\begin{definition}{\label{can_cover}}
Let $P \in X$ be an slc surface germ of index $N$. The \emph{canonical covering} $\pi: Z \rightarrow X$ is defined by
\begin{equation*}
Z= \Spec_X (\OO_X \oplus \OO_X(K_X) \oplus \cdots \oplus \OO_X((N-1)K_X)),
\end{equation*}
where the multiplication structure is determined by a choice of isomorphism $\OO_X(NK_X) \cong \OO_X$. 
\end{definition}

We will also use the terminology \emph{index one cover} to express the same idea (recall that $K_Z$ is Cartier, hence has index $1$). Let $\xi_N$ be a primitive $N^\text{th}$ root of unity. There is a natural $\mu_N$ action on each $\OO_X(iK_X)$ given by multiplication by $\xi^i_N$, and we note that the canonical covering morphism $\pi$ is a cyclic quotient of degree $N$ by the induced action on $Z$. 

\begin{definition}\label{qg}
Let $(P \in X, D)$ be the germ of a stable pair, $N$ be the index of $X$, $Z \rightarrow X$ the canonical covering, and $D_Z$ the inverse image of $D$. We say that a deformation $(\X, \D)/S$ of $(X, D)$ is \emph{$\Q$-Gorenstein} if there is a $\mu_N$-equivariant deformation $(\ZZ, \D_\ZZ)/S$ of $(Z, D_Z)$ extending the natural $\mu_N$ action on $Z$ whose $\mu_N$ quotient is $(\X, \D)/S$. 
\end{definition}

\begin{remark} \label{index}
We say that $(X,D)$ satisfies the \emph{index condition} if the divisorial pullback of $D$ to the canonical covering at every surface germ of $X$ is Cartier. For stable pairs of type $(m,n)=(1,2)$, this condition is vacuous. See {\cite[Definition 2.4]{DH18}} for more details. We note that Definition~\ref{qg} is equivalent to conditions $(1)$ and $(2)$ of Definition~\ref{qg1} if the index condition holds.

\end{remark}

It is clear how to modify the definition of $\Q$-Gorenstein family in the context of surfaces (no marked curve) or surface germs: simply forget all conditions involving the marked divisor. 

\begin{definition}
We say that a stable pair $(X,D)$ is \emph{smoothable} if there is a $\Q$-Gorenstein deformation $(\X, \D)/\Delta$ of $(X,D)$ over the germ of a smooth curve such that the generic fiber $\X_\eta$ of $\X/\Delta$ is smooth. 
\end{definition}

It follows from parts $(2)$ and $(3)$ of Definition~\ref{stable_pair} that for a stable pair $(X,D)$, the divisors $-K_X$ and $D$ are both ample. In particular, if $(X,D)$ is smoothable, $X$ must smooth to a del Pezzo surface.

\begin{theorem}[{\cite[Theorem 2.5]{DH18}}]
There is a Deligne-Mumford stack $\mathfrak{F}$ whose objects are $\mathbb{Q}$-Gorenstein families of stable pairs of type $(m,n)$ satisfying the index condition.
\end{theorem}

\begin{definition} \label{aux_stacks}
Fix $(m,n)=(1,2)$. Let $\mathfrak{F}_{K^2=5} \subset \mathfrak{F}$ be the open and closed substack parametrizing stable pairs $(X,D)$ with $K_X^2=5$. Let $\mathfrak{P}$ denote the component of $\mathfrak{F}_{K^2=5}$ whose general point is a pair $(\Sigma_5, C)$ where $C$ is smooth of class $-2K_{\Sigma_5}$.  
\end{definition} 

We note that $\mathfrak{F}_{K^2=5}$ is in fact an open and closed substack of $\mathfrak{F}$ since $K^2$ (and moreover $(K+D)^2$) is constant in $\mathbb{Q}$-Gorenstein families of stable pairs (see, for example, \cite{Has99}). Also, now it makes sense to define the open substacks $\mfp_0$ and $\mfp_0^\text{sm}$ of Definition~\ref{stacks} (the openness follows from the fact that the set of allowed singularity types for pairs in these substacks is closed under $\mathbb{Q}$-Gorenstein deformation).

\begin{remark}
The properness of the stack $\mathfrak{F}$ in general is a delicate issue. There is a partial valuative criterion properness proven in \cite[Proposition 2.11]{DH18}: Up to base change, a family over a DVR with smooth generic fiber can be completed to a family where the special fiber is a stable pair. Moreover, this new family will be $\mathbb{Q}$-Gorenstein if the special fiber satisfies the index condition (recall Remark~\ref{index}), but there is no reason that the index condition should hold a priori. However, as noted in Remark~\ref{index}, for stable pairs of type $(m,n)=(1,2)$, the index condition holds vacuously. Hence $\mfp$ is proper.
\end{remark}

We will now give some properties of stable pairs of type $(m,n)$ and their families. We begin with a description of some singularities that arise on stable pairs and conclude with a smoothability criterion for pairs with such singularities. 

\begin{definition}
Fix co-prime positive integers $a$ and $r$ with $a<r$. Let $\Z/r\Z$ act on $\C^2$ via the diagonal matrix
$$\begin{pmatrix}
\xi_r & 0 \\ 0 & \xi_r^a
\end{pmatrix},$$
where $\xi_r$ is a primitive $r^{\text{th}}$ root of unity. The resulting singularity is called a \emph{cyclic quotient singularity of type $\frac{1}{r}(1,a)$}.
\end{definition}

Such singularities are uniquely determined by their minimal resolutions. The exceptional locus of the minimal resolution of a cyclic quotient singularity of type $\frac{1}{r}(1, a)$ is a chain of rational curves $E_1, \dots, E_n$ with self-intersections $E^2_i=-c_i<0$ for all $i$. The $c_i$ can be computed via the continued fraction 
\begin{equation} \label{cfrac}
\frac{r}{a}=c_1 - \frac{1}{c_2-\frac{1}{c_3-\dots}}.
\end{equation}

Conversely, given $E_i$, $c_i$, and a continued fraction representation as in ~(\ref{cfrac}), we say that the singularity created by contracting the $E_i$ is of type $\frac{1}{r}(1,a)$. We remark that this notation depends on one of the two possible orderings of the $E_i$.

\begin{definition}[{\cite[Definition 3.7]{KSB88}}]
\label{classT}
A surface singularity is said to be of \emph{class $T$} if it is a cyclic quotient singularity and admits a $\Q$-Gorenstein one-parameter smoothing. 
\end{definition}

In the definition of class $T$ given in \cite{KSB88}, a deformation $\X/S$ is said to be $\Q$-Gorenstein if $K_\X$ is $\Q$-Cartier. This is an a priori weaker condition than the notion of $\Q$-Gorenstein given in Definition~\ref{qg}. However, as remarked in {\cite[Section 2.1]{HP10}}, the two notions coincide when the central fiber $X$ has quotient singularities and the base $S$ is a smooth curve. There is a well known classification of class $T$ singularities due to Koll\'ar and Shepherd-Barron which we now present.

\begin{proposition}[{\cite[Proposition 3.10]{KSB88}}]\label{Classt}
A class $T$ singularity is either a rational double point (ADE, du Val) or a cyclic quotient singularity of type 
\begin{equation} \label{p,q}
\frac{1}{p^2q}(1, dpq-1)
\end{equation}
where $p,q$ are integers and $d$ is co-prime to $p$. 
\end{proposition}

We make a few remarks about class $T$ singularities. Class $T$ singularities are precisely the log terminal $\Q$-Gorenstein-smoothable surface singularities ({\cite[Theorem 3.4]{Pr17}}). For a given class $T$ singularity, there is an irreducible component of its deformation space parametrizing $\Q$-Gorenstein deformations. Hence, $\Q$-Gorenstein deformations of class $T$ singularities are class $T$ ({\cite[Theorem 3.9, Section 7]{KSB88}}). A non-du Val class $T$ singularity of the form in ~(\ref{p,q}) has index $p$ and canonical cover of type $A_{pq-1}$. The $\mu_p$ action on the equation 
\begin{equation} \label{Apq}
f=xy+z^{pq}=0
\end{equation}
is given by
\begin{equation} \label{can_action}
(x,y,z) \mapsto (\xi x, \xi^{-1} y, \xi^d z)
\end{equation}
(see the remarks immediately following Proposition 5.3 in \cite{BR95}, for example). We will also need to make use of the following theorem.

\begin{theorem}[{\cite[Theorem 3.1]{HP10}}] \label{lc_big}
Let $X$ be a projective surface with log canonical singularities such that $-K_X$ is big. Then there are no local-to-global obstructions to deformations of $X$. In particular, if the singularities of $X$ admit $\mathbb{Q}$-Gorenstein smoothings, then $X$ admits a $\mathbb{Q}$-Gorenstein smoothing.
\end{theorem} 

Before establishing the smoothability criterion, we need two important facts.

\begin{lemma} \label{lift}
Let $(X,D)$ be a stable pair of type $(m,n)$ such that $X$ has class $T$ singularities and $D$ is Cartier. Then $H^1(\OO_D(D))=0$. 
\end{lemma}

\begin{proof}
By Lemma $3.14$ in \cite{Hac04}, $H^1(\OO_X(D))=0$ since $X$ is log terminal (this is a consequence of Kodaira vanishing). Now, the exact sequence
\begin{equation*}
0 \rightarrow \OO_X \rightarrow \OO_X(D) \rightarrow \OO_D(D) \rightarrow 0
\end{equation*}
induces a long exact sequence in cohomology
\begin{equation*}
\cdots \rightarrow H^1(\OO_X(D)) \rightarrow H^1(\OO_D(D)) \rightarrow H^2(\OO_X) \rightarrow \cdots
\end{equation*}
By Serre duality, $H^2(\OO_X)=H^0(K_X)^\vee=0$ since $K_X$ is anti-ample. The result is immediate. 
\end{proof}

\begin{lemma}[{\cite[Lemma 5.5]{Hac01}}] \label{Pic}
Let $X$ be a surface with log canonical and $\Q$-Gorenstein smoothable singularities with $-K_X$ ample, and let $\X/\Delta$ be a deformation of $X$ over the germ of a  smooth curve. Then the restriction map
\begin{equation*}
\Pic \X \rightarrow \Pic X
\end{equation*}
is an isomorphism.
\end{lemma}

The proof mimics a portion of the proof of the proposition cited. We include it for the reader's convenience.

\begin{proof}
We have a commutative diagram
\begin{equation*}
\begin{tikzcd}
\Pic \X \arrow[r] \arrow[d]
	& \Pic X \arrow[d] \\
H^2(\X, \Z) \arrow[r] 
	& H^2(X, \Z)	
\end{tikzcd}
\end{equation*}

The restriction map $H^2(\X, \Z) \rightarrow H^2(X, \Z)$ is an isomorphism because $X$ is a homotopy retract of $\X$. The map $\Pic X \rightarrow H^2(X, \Z)$ fits into the long exact sequence in cohomology
\begin{equation*}
\cdots \rightarrow H^1(\OO_X) \rightarrow \Pic X \rightarrow H^2(X, \Z) \rightarrow H^2(\OO_X) \rightarrow \cdots
\end{equation*}
associated to the exponential sequence. Using Serre duality and the fact that $-K_X$ is ample, we see that $H^2(\OO_X)=0$ as in the proof of Lemma~\ref{lift}. 

By Theorem~\ref{lc_big}, $X$ admits a one-parameter smoothing over the germ of a smooth curve to a del Pezzo surface $Y$. Since $\chi(\OO_Y)=1$, we must have $H^1(\OO_X)=0$. Hence the map $\Pic X \rightarrow H^2(X, \Z)$ is an isomorphism. 

Since $H^1(\OO_X)=H^2(\OO_X)=0$, by cohomology and base change, $R^1f_\ast \OO_\X=R^2f_\ast \OO_\X=0$. Since $\Delta$ is affine, $H^1(\OO_\X)=H^2(\OO_\X)=0$ (see {\cite[Theorem III.3.7, Exercise III.8.1, Theorem III.12.11]{Har77}}). By considering the exponential sequence as before, we see that the map $\Pic \X \rightarrow H^2(\X, \Z)$ is an isomorphism. Therefore, the restriction map $\Pic \X \rightarrow \Pic X$ is also an isomorphism.

\end{proof}

We now present the main theorem of this section. As in the previous lemma, let $\Delta$ denote the germ of a smooth curve.

\begin{theorem} \label{smooth}
Let $(X,D)$ be an slc stable pair such that $X$ has class $T$ singularities and $D$ is Cartier. Then the following hold:
\begin{enumerate} 
\item $(X,D)$ is smoothable.
\item The generic fiber of any $\Q$-Gorenstein deformation of $(X,D)$ over $\Delta$ is smoothable.
\item Any $\Q$-Gorenstein deformation of the singularities of $X$ over $\Delta$ can be realized on a stable pair.
\end{enumerate}
\end{theorem}

\begin{proof}
By Theorem~\ref{lc_big}, a $\Q$-Gorenstein smoothing of the singularities of $X$ over $\Delta$ lifts to a $\Q$-Gorenstein smoothing $\X/\Delta$ of $X$. By Lemma~\ref{lift}, $H^1(\OO_D(D))=0$, hence by {\cite[Corollary 3.2]{Has99}}, we can lift this family of surfaces in turn to a family of slc pairs $(\X, \D)/\Delta$ satisfying all the conditions of a $\Q$-Gorenstein family except (a priori) that the generic fiber is a stable pair. 

Since $nK_X +mD \sim 0$, by Proposition~\ref{Pic}, we have the relation $nK_\X+m\D \sim 0$. Therefore, $nK_{\X_\eta}+m\D_\eta \sim 0$ by restriction. We also note that $\chi(\OO_{\X_\eta})=1$, since $\chi(\OO_X)=1$. Now, fix $\epsilon$ such that $K_X+(m/n+\epsilon)D$ is ample. Passing to a sufficiently high multiple $N$ such that $N(K_\X+(m/n+\epsilon) \D)$ is Cartier and restricting to the generic fiber shows that $K_{\X_\eta}+(m/n+\epsilon)\D_\eta$ is ample as well. This concludes the proof of $(1)$.

Since the hypotheses of the theorem are preserved under any $\Q$-Gorenstein deformation over $\Delta$, $(2)$ is immediate.

For $(3)$, lift a $\Q$-Gorenstein deformation of the singularities of $X$ to a $\Q$-Gorenstein deformation of slc pairs as above. Repeating the argument in the proof of $(1)$ shows that the generic fiber is also a stable pair.\end{proof}

\begin{remark}
The first claim in the theorem can be strengthened slightly: We may assume the singularities of $X$ are log canonical and $\Q$-Gorenstein-smoothable, if we also require that $H^1(\OO_D(D))=0$. 
\end{remark}

\section{Proof of main result: resolving $\varphi$}
\numberwithin{theorem}{section}
\counterwithin{theorem}{subsection}
\numberwithin{equation}{section}
\counterwithin{equation}{subsection}
This section will be devoted to proving Theorem~\ref{mainthm}. Let $\mathfrak{P}$ denote the moduli space in Definition~\ref{aux_stacks}. Again, for brevity, we will simply use the terminology ``stable pairs" (omitting ``of type $(m,n)$"). 

Using the stratifications in Section $2$, for each marked smooth plane quintic or trigonal curve $(D, E)$, we exhibit a stable pair $(X,D)$. For each bielliptic curve $D$, we exhibit a stable pair $(X, D)$. Stability of these pairs will be addressed in Proposition~\ref{stable_prop}. In this section, we also explain how to address the hyperelliptic curves. Throughout, we will abuse notation and write $D$ for both the curve that we start with and its image in any birational model of the surface into which $D$ naturally embeds. 

\subsection{Marked plane quintics}\label{pl_q} 

For a given marked plane quintic curve $(D, E)$ of type $(a_1, \dots, a_5)$, choose a line $\ell$ in $\PP^2$ such that 
\begin{equation*}
E=\ell|_D=\displaystyle \sum_i a_ip_i.
\end{equation*}
Separate $D$ from $\ell$ by blowing up, and contract the strict transform of $\ell$ and any exceptional curves of self-intersection strictly less than $-1$. We obtain a surface $X$ with singularity type
\begin{equation*}
\frac{1}{4}(1,1) \displaystyle \oplus \bigoplus_{a_i>1} A_{a_i-1}.
\end{equation*}
We have constructed the desired pair $(X, D)$.  

\subsection{Marked trigonal curves, type $(0; b_1, b_2, b_3, b_4)$} \label{M=0;4} 

Given a marked trigonal curve of genus six and Maroni invariant $0$ denoted $(D, E)$, embed $D$ in $\PP^1 \times \PP^1$. The curve $D$ has class $3e+4f$. Choose a particular ruling $e_0 \in |e|$ such that $E=e_0|_D$ on $D$. Separate $D$ from $e_0$ by blowing up. Contracting the strict transform of $e_0$ and all exceptional curves of self-intersection strictly less than $-1$ yields a surface $X$ with singularity type 
\begin{equation*}
\frac{1}{4}(1,1) \displaystyle \oplus \bigoplus_{b_i>1} A_{b_i-1}.
\end{equation*}
We have constructed the desired pair $(X, D)$.

\subsection{Marked trigonal curves, type $(2; [a_1], a_2, a_3, a_4)$}\label{M=2} For a given pair $(D, E)$ of this type, embed $D$ in $\F_2$. Let $e$ denote the negative section and choose $f$ such that $E=(e+f)|_D$. Separate $D$ from $e \cup f$ by blowing up. If necessary (this will depend on whether $a_1=1$ or $a_1>1$), further separate $D$ from the chain of curves connecting the strict transforms of $e$ and $f$ by blowing up. This process yields a chain $C$ of rational curves of self-intersection 
\begin{equation*}
[-3, \underbrace{-2, \dots, -2}_{\text{$a_1-1$}} , -3].
\end{equation*}
Contracting $C$ along with any exceptional curves of self-intersection strictly less than $-1$ produces a surface $X$ with singularity type
\begin{equation*}
\frac{1}{4(a_1+1)}(1, 2a_1+1) \oplus \displaystyle \bigoplus_{ \underset{i\neq 1}{a_i >1}} A_{a_i-1}.
\end{equation*}
The quotient singularity of $X$ is index two class $T$. We have constructed the desired pair $(X, D)$.

\subsection{Bielliptic curves}\label{biell} As noted in Section $2$, such a curve $D$ embeds as a quadric section of an elliptic cone $X$ of degree $5$ in $\PP^5$, hence $D$ necessarily has class $-2K_ X$ (which is ample). Moreover, $H^1(\OO_D(D))=0$ by Serre duality.  Since $X$ is log canonical and $D$ avoids the singularity, $(X,D)$ is a smoothable slc stable pair by Theorem~\ref{smooth}. Note that any smoothing of the elliptic singularity is automatically $\Q$-Gorenstein, since the singularity is Gorenstein. Since $K_X^2=5$, the pair smooths to $(\Sigma_5, C)$ where $C$ is smooth of class $-2K_{\Sigma_5}$, as desired.

\subsection{Hyperelliptic curves}\label{hyper} There is a complete list of ADE-singular plane sextic curves given in \cite[Table 2]{Ya96}. In particular, we can find such a curve with an $A_{13}$ singularity and four nodes in general position. Blow up the four nodes to recover $\Sigma_5$, and let $D$ be the strict transform of the sextic. By construction, $D$ has class $-2K_{\Sigma_5}$. Stable reduction of a curve with an $A_{13}$ singularity yields a smooth genus six hyperelliptic curve. Moreover, every such curve arises in this way (see {\cite[Example 6.2.1]{Has00}}). It follows immediately from Definition~\ref{stable_pair} that the pair $(\Sigma_5, D)$ is stable. By deforming the $A_{13}$ curve in $\Sigma_5$, we obtain a $\mathbb{Q}$-Gorenstein smoothing of this pair to $(\Sigma_5, C)$ where $C$ is smooth of class $-2K_{\Sigma_5}$.

\subsection{Proof of main theorem} We have completed the list of pairs necessary to prove Theorem~\ref{mainthm}. The following sequence of propositions will constitute our proof.

\begin{proposition}\label{stable_prop}
All of the pairs constructed in (\ref{pl_q})~--~(\ref{hyper}) are smoothable stable pairs of type $(1,2)$. Moreover, all of these pairs lie in $\mathfrak{\mfp_0}$.
\end{proposition}

\begin{proof}
We outline the general technique for showing that each pair over the trigonal curves and plane quintics lies in $\mathfrak{\mfp_0}$ below. Note that we have already addressed the pairs associated to the bielliptic and hyperelliptic curves in ~(\ref{biell}) and ~(\ref{hyper}).

Fix one of these pairs $(X,D)$ such that $D$ is a smooth plane quintic or trigonal curve. Let $\phi: X' \rightarrow X$ be the minimal resolution. We have seen $X'$ can be realized as a sequence of blow-ups of a smooth surface in which $D$ naturally embeds and whose intersection theory is well understood (see Section $2$). As a result, there is a natural set of generators for $\Pic X'$. We have seen that $X$ has index $2$ class $T$ singularities (and potentially also has type $A$ singularities) and is in particular $\Q$-factorial. We can express $\phi^\ast(-2K_X)$ and $\phi^*(D)$ in terms of these Picard generators, and we determine that they are linearly equivalent. Moreover, $\phi^*(D)$ coincides with $D'$ (the strict transform of $D$), since $D$ avoids the singularities of $X$. By the projection formula, we obtain $D=-2K_X$. This computation also verifies that $-K_X$ is ample; one checks that $\phi^*(-K_X)$ is nef and trivial precisely along curves contracted by $\phi$. We also see that $(K_X)^2=5$. Moreover, since $X$ is log terminal and $D$ avoids singularities, $(X,D)$ is slc. Combining all of this, we see that $(X,D)$ satisfies the hypotheses of Theorem~\ref{smooth}. Thus, $(X,D)$ smooths to $(\Sigma_5, C)$, where $C$ is smooth of class $-2K_{\Sigma_5}$ as desired.
\end{proof}

\begin{Example}[\emph{Marked plane quintics, type $(1,1,1,1,1)$}] \label{pl_q_2}
Let $D$ be a smooth plane quintic and let $\ell$ be a line transverse to $D$. Let $\pi_1: X' \rightarrow \PP^2$ be the blow-up of $\PP^2$ at the $5$ points of intersection of $D$ and $\ell$, and let $\pi_2: X' \rightarrow X$ denote the contraction of $\ell'$ (the strict transform of $\ell$). Let $L$ denote the hyperplane class on $\PP^2$, let $G_i$ be the five $\pi_1$-exceptional divisors, let $D'$ be the strict transform of $D$ under $\pi_1$, and let $D''$ denote its image in $X$. We show that $(X, D'')$ is a stable pair satisfying the hypotheses of Theorem~\ref{smooth} with $K_X^2=5$, hence this pair lies in $\mathfrak{P}$.

We compute
\begin{equation*}
\pi_2^\ast(-2K_X)=-2K_{X'}-\ell'=\pi_1^\ast(5L)-\underset{i=1}{\sum^5} G_i=D'=\pi_2^\ast (D''),
\end{equation*}
hence by the projection formula,
\begin{equation*}
D''=-2K_{X}.
\end{equation*}

To verify that $K_{X}+D''=-K_X$ is ample, we choose
\begin{equation*}
\ell'=\pi_1^\ast L - \underset{i=1}{\sum^5} G_i
\end{equation*}
and the $G_i$ as Picard generators for $X'$. Fix an irreducible curve $C \subset X$; we need to show that this curve is positive against $-K_{X}$. Since $X$ is $\Q$-factorial, we can pull back to the minimal resolution to compute intersection numbers. If $C'$ (the strict transform of $C$ under $\pi_2$) is not $\ell'$ or any of the $G_i$, it is non-negative along each. By non-degeneracy of the intersection pairing on $X'$, in fact $C'$ must be strictly positive along at least one of them. Since we can write $\pi_2^\ast(-K_{X})$ as a positive linear combination of $\ell'$ and the $G_i$, it follows that we only need to check how this pullback intersects each of them. Ampleness of $-K_{X}$ is immediate. 

The pair $(X, D'')$ is slc since $X$ is log terminal (class $T$) and $D''$ avoids singularities. We conclude that $(X, D'')$ is an slc stable pair. Moreover, $K_{X}^2=5$ and the pair satisfies the hypotheses of Theorem~\ref{smooth}. 
\end{Example}

\begin{Example}[\emph{Marked trigonal curves, type $(2; [4])$}]\label{M=2, ex}

Let $D$ be such a curve in $\F_2$, let $e$ denote the negative section, and let $f_p$ denote the distinguished fiber (see Section $2$). Let $\phi_1: X' \rightarrow \F_2$ denote the sequence of blow-ups described in ~(\ref{M=2}), and let $G_i$ be the $\phi_1$-exceptional divisors ($i=1, \dots, 4$). Let $\phi_2: X' \rightarrow X$ be the minimal resolution of $X$. Let $D'$ be the strict transform of $D$ under $\phi_1$, and let $D''$ be its image in $X$. We show that the pair $(X, D'')$ is stable and satisfies the hypotheses of Theorem~\ref{smooth} with $K_X^2=5$, hence this pair lies in $\mathfrak{P}$.

We compute
\begin{equation}\label{4.1}
\phi_1^\ast(K_{\F_2})=K_{X'}-G_1-2G_2-3G_3-4G_4.
\end{equation}
On the other hand,
\begin{equation}\label{4.2}
\phi_1^\ast(K_{\F_2})=\phi_1^\ast(-2e-4f_p)=-2e'-4f_p'-6G_1-10G_2-14G_3-14G_4,
\end{equation}
hence
\begin{equation}\label{4.3}
K_{X'}=-2e'-4f'_p-5G_1-8G_2-11G_3-10G_4.
\end{equation}
In ~(\ref{4.2}), ~(\ref{4.3}), $e'$ and $f'_p$ refer to the strict transforms of $e$ and $f_p$ under $\phi_1$. 

Next, using ~(\ref{4.3}), we see that
\begin{equation}\label{4.4}
\phi_2^\ast(-2K_{X})=3e'+7f_p'+9G_1+15G_2+21G_3+20G_4.
\end{equation}
Also, since $D''$ avoids the singularities of $X$, $\phi_2^\ast (D'')=D'$. We compute
\begin{equation}\label{4.5}
\phi_1^\ast(D)=D'+G_1+2G_2+3G_3+4G_4.
\end{equation}
On the other hand,
\begin{equation}\label{4.6}
\phi_1^\ast(D)=\phi_1^\ast(3e+7f_p)=3e'+7f_p'+10G_1+17G_2+24G_3+24G_4.
\end{equation}
Therefore, by combining ~(\ref{4.4}), ~(\ref{4.5}), and ~(\ref{4.6}),
\begin{equation}\label{4.7}
D'=3e'+7f_p'+9G_1+15G_2+21G_3+20G_4=\phi_2^\ast(-2K_{X}).
\end{equation}
By the projection formula, $D''=-2K_{X}$. 

To verify ampleness of $K_{X}+D''=-K_{X}$, we choose $e', f'_p$ and the $G_i$ as Picard generators for $X'$. An analogous argument to that in Example~\ref{pl_q_2}  implies that we only need to check that $-K_X$ is positive against $G_4$, which it is. We again note that the pair $(X, D'')$ is slc since $X$ is log terminal and $D''$ avoids singularities. Hence this pair is in fact an slc stable pair. Moreover, $K_{X}^2=5$ and the pair satisfies the hypotheses of Theorem~\ref{smooth}.

\end{Example}
\begin{remark}\label{bdry_rmk}
We note that the pairs $(X, D)$ associated to plane quintics (Subsection~\ref{pl_q}) lie in the boundary locus $\ZZ_1$ described in Section $1$ (which is now well-defined as a result of Proposition~\ref{stable_prop}). We remark that it follows from the construction of $(X, D)$ that $\ZZ_1$ is in fact a divisor: The locus of plane quintics in $\mathcal{M}_6$ is of dimension $12$, and the moduli space of $5$ points on $\PP^1$ (the points of intersection between a quintic and a line) is of dimension $2$. We also see from this discussion that the fiber of the forgetting map $j: \mfp_0 \dashrightarrow \overline{\mathcal{M}}_6$ over a smooth plane quintic curve is $2$ dimensional.

Similarly, the pairs $(X, D)$ associated to trigonal curves (Subsection~\ref{M=0;4} and Subsection~\ref{M=2}) lie in the boundary locus $\ZZ_2$ described in Section $1$. We remark that $\ZZ_2$ is in fact a divisor: The trigonal locus in $\mathcal{M}_6$ is of dimension $13$, and the moduli of $4$ points on $\PP^1$ (the points of intersection between a trigonal curve with $M=0$ on $\PP^1 \times \PP^1$ and the appropriate ruling) is of dimension $1$. We also see from this discussion that the fiber of the forgetting map $j: \mfp_0 \dashrightarrow \overline{\mathcal{M}}_6$ over a smooth trigonal curve of genus six is $1$ dimensional.

By definition, the pairs $(X, D)$ in this subsection lie in the boundary locus $\ZZ_3 \subset \mfp_0$. Since the bielliptic locus in $\mathcal{M}_6$ is of dimension $10$ and the bielliptic involution is unique (recall Subsection~\ref{biell_geo}), the locus $\ZZ_3$ is in fact $10$ dimensional as asserted in Section $1$.

\end{remark}

\begin{proposition}
The stack $\mfp_0$ is smooth and Deligne-Mumford.\end{proposition}

\begin{proof}
Since $\mathfrak{F}$ is Deligne-Mumford, so is $\mfp_0$. 

The $\mathbb{Q}$-Gorenstein deformation space of any singularity allowed on a surface in $\mfp_0$ is smooth, since the possible singularities are du Val, cyclic quotient, or simple elliptic of degree $5$. The smoothness of the deformation space of this simple elliptic singularity is proven in \cite[Section 9.2(b)]{Pi74}. Every deformation of this elliptic singularity is $\mathbb{Q}$-Gorenstein since this singularity is Gorenstein. There are no local-to-global obstructions for deformations of any of the surfaces in $\mf$ by Theorem~\ref{lc_big}. By {\cite[Proposition 3.3]{Has99}}, the $\mathbb{Q}$-Gorenstein deformation space of any pair in $\mfp^\text{sm}_0$ is smooth.

Note that a pair $(X, D)$ in $\mfp_0$ where $D$ has ADE singularities is not slc, but the conclusion of {\cite[Proposition 3.3]{Has99}} still holds since we require that $D$ avoids the singularities of $X$. Hence the deformation space of such a pair is smooth. We conclude that $\mfp_0$ is smooth.

\end{proof}

\begin{proposition} \label{K3}
For any pair $(X, D)$ in $\mfp_0$, the double cover of $X$ branched along $D$ is a $K3$ surface with Gorenstein slc singularities.
\end{proposition}

\begin{proof}
Since $D=-2K_X$, the double cover of $X$ branched along $D$ is, by definition,
\begin{equation*}
X^{(2)}=\Spec_X(\OO_X \oplus \OO_X(K_X)), 
\end{equation*}
where the $\OO_X$-algebra structure is determined by multiplication by $D$. Let $\pi: X^{(2)} \rightarrow X$ be the natural morphism. 

By definition of $\mfp_0$, the surface $X^{(2)}$ has only combinations of du Val and simple elliptic singularities (including ``empty" combinations).

For any pair $(X, D)$ in $\mfp_0$, by adjunction and the fact that $D \sim -2K_X$, the line bundle $\omega_{X^{(2)}}$ is trivial. Moreover, 
\begin{equation*}
h^1(X^{(2)}, \OO_{X^{(2)}})=h^1(X,\OO_X) + h^1(X, \OO_X(K_X))=0. 
\end{equation*}
Thus, $X^{(2)}$ is a $K3$ surface with Gorenstein slc singularities as claimed.

\end{proof}

\begin{remark} \label{notable}
We comment on some notable aspects of the $K3$ surfaces associated to the pairs $(X, D)$ containing plane quintic and trigonal curves. For a generic pair $(X, D)$ associated to plane quintics in Subsection~\ref{pl_q} (the general member of $\ZZ_1$), the double cover of $X$ branched along $D$ is a $K3$ surface with an $A_1$ singularity. This $K3$ surface corresponds to the generic point of the component of the discriminant divisor $\mathcal{H}_2$ described by Artebani and Kond\=o. Moreover, the lattice-polarization in this case is isomorphic to $U(2) \oplus D_4$ (see \cite[Section $3$]{AK11}), which is in particular of rank $6$. 

This $K3$ surface can also be constructed by taking the minimal resolution of the double cover of $\PP^2$ branched along the union of $D$ and $\ell$ and contracting the strict transform of $\ell$ (a $(-2)$-curve). We note that constructing periods for pairs $(D, \ell)$ via such double covers has already been considered in full detail in \cite{Laz09} and is also mentioned in \cite{AK11}.

For a generic pair $(X, D)$ associated to trigonal curves in Subsection~\ref{M=0;4} (the general member of $\ZZ_2$), the double cover of $X$ branched along $D$ is a $K3$ surface with an $A_1$ singularity. This $K3$ surface is the generic point of the component of the discriminant divisor $\mathcal{H}_3$ described by Artebani and Kond\=o. Moreover, the lattice polarization in this case is isomorphic to $U \oplus A_1^{\oplus 4}$ (see \cite[Section 3]{AK11}), which is in particular of rank $6$.

A complete description of the singularities appearing on the $K3$ surfaces associated to the pairs constructed in this paper appears in Table~\ref{table:pairs}.
\end{remark}

\begin{proposition}\label{period}
There is a period map 
\begin{equation*}
\tilde \varphi: \mfp_0 \rightarrow (D/\Gamma)^\ast.
\end{equation*}
\end{proposition}

\begin{proof}
Consider the tautological family $\mathfrak{W} \rightarrow \mfp_0$. Taking the double cover of $\mathfrak{W}$ branched along the marked curves yields a family whose fibers parametrize $K3$ surfaces, except over pairs with elliptic singularities. Hence we have a rational period map
\begin{equation*}
\tilde \varphi : \mfp_0 \dashrightarrow (D/\Gamma)^\ast
\end{equation*}
defined away from the elliptic cone pairs. Given a smoothing of such a pair over a germ of a smooth curve, this period map uniquely extends over the closed point. Since the double cover of any pair with elliptic singularities (in fact any pair in $\mfp_0$ by Proposition~\ref{K3}) has insignificant limit singularities (see \cite[Theorem 1]{Sh79}), this extension in fact does not depend on the smoothing. See, for example, the discussion in {\cite[Section 3.3]{LO16}}. Since $\mfp_0$ is smooth, this rational period map extends to a morphism
\begin{equation*}
\tilde \varphi: \mfp_0 \rightarrow (D/\Gamma)^\ast
\end{equation*}
as claimed. The image of pairs with elliptic singularities (including the elliptic cones of Subsection~\ref{biell}) lies in the boundary of $(D/\Gamma)^\ast$.
\end{proof}

\begin{remark}
We see that, as expected, $\tilde \varphi$ can have positive dimensional fibers: We know that $\ZZ_3$ is $10$ dimensional in $\mfp_0$ (recall Remark~\ref{bdry_rmk}), and by \cite[Remark $4.7$]{AK11}, bielliptic curves are mapped to a $1$ dimensional boundary component of $(D/\Gamma)^\ast$. We note for completeness that the boundary of $(D/\Gamma)^\ast$ has $2$ zero dimensional components and $14$ one dimensional components (\cite[Corollary $4.2$, Theorem $4.5$]{AK11}), although the component corresponding to bielliptic curves is the only part of the boundary we consider in this paper.
\end{remark}

\begin{proposition}

The natural birational forgetting map $j$ restricts to a surjective morphism
\begin{equation*}
j|_{\mfp_0^\text{sm}}: \mfp_0^\text{sm} \surj \mathcal{M}_6 \setminus \mathcal{H}_6,
\end{equation*}
where $\mathcal{H}_6$ denotes the hyperelliptic locus. 

\end{proposition}

\begin{proof}
By the explicit construction of the pairs in (\ref{pl_q})~--~(\ref{hyper}) and the definition of $\mfp_0^\text{sm}$ (Definition~\ref{stacks}), every smooth genus six non-hyperelliptic curve arises on a pair in $\mfp_0^\text{sm}$, and conversely, every curve on a pair in $\mfp_0^\text{sm}$ is smooth and of genus six. This verifies the claim that $j$ restricts to a surjection of $\mfp_0^\text{sm}$ onto $\mathcal{M}_6 \setminus \mathcal{H}_6$.

\end{proof}

\begin{remark} \label{sim_st}
Consider the tautological family $\mathfrak{W} \rightarrow \mfp_0$ as in the proof of Proposition~\ref{period}. We can construct a diagram

\begin{center}
\begin{tikzcd}[column sep=small]
\tilde \mfp_0 \arrow["\tilde j", d, swap] \arrow["\nu", r]  & \mfp_0 \arrow["j", dl, dashed, swap] & \\
\overline {\mathcal{M}}_6 
\end{tikzcd}
\end{center}
via simultaneous stable reduction (\cite[Theorem 3.5, Corollary 6.3]{CML13}). Moreover, the image of $\tilde j$ is a partial compactification of ${\mathcal{M}}_6$. As noted previously, every hyperelliptic curve can be realized via stable reduction of a curve with an $A_{13}$ singularity. Therefore, the image of $\tilde j$ contains $\mathcal{M}_6$. By definition, $\mfp_0$ contains pairs of the form $(\Sigma_5, C)$, where $C$ has ADE singularities. Stable reductions of such curves may (and will) be nodal; for example, consider cuspidal curves (these also exist on $\Sigma_5$ by the results in \cite[Table 2]{Ya96}). The image of $\tilde j$ consequently intersects the boundary of $\overline {\mathcal{M}}_6$. Note that $\tilde \mfp_0$ resolves the indeterminacy of $j$, but it is not immediately clear how to describe this space in a modular way over pairs containing singular curves.
\end{remark}

Table~\ref{table:pairs} below summarizes the pairs constructed in this section and their associated $K3$ surfaces.

\begin{table}[ht]
%\caption{Pairs $(X, D)$ constructed in subsections (\ref{pl_q})~--~(\ref{hyper}) and their associated $K3$ surfaces.}
\centering 
\begin{tabular}{c c c c} 
\hline\hline 
$D$ & $X_{\text{sing}}$ & $\ZZ_i$ & Singularities of $K3$ \\ [0.5ex] 
%heading
\hline 
Plane quintic of type $(a_1, \dots, a_5)$ & $\frac{1}{4}(1,1) \displaystyle \oplus \bigoplus_{a_i>1} A_{a_i-1}$ & $\ZZ_1$ & $A_1 \displaystyle \oplus \bigoplus_{a_i>1} A^{\oplus 2}_{a_i-1}$ \\ 
Trigonal of type $(0; b_1, b_2, b_3, b_4)$ & $\frac{1}{4}(1,1) \displaystyle \oplus \bigoplus_{b_i>1} A_{b_i-1}$ & $\ZZ_2$ & $A_1 \displaystyle \oplus \bigoplus_{b_i>1} A^{\oplus 2}_{b_i-1}$  \\
Trigonal of type $(2; [a_1], a_2, a_3, a_4)$ & $\frac{1}{4(a_1+1)}(1, 2a_1+1) \displaystyle \oplus \bigoplus_{ \underset{i\neq 1}{a_i >1}} A_{a_i-1}$ & $\ZZ_2$ & $A_{2a_1+1} \displaystyle \oplus \bigoplus_{ \underset{i\neq 1}{a_i >1}} A^{\oplus 2}_{a_i-1}$ \\
Bielliptic & Simple elliptic & $\ZZ_3$ & Simple elliptic \\
$\sim -2K_{\Sigma_5}$ with an $A_{13}$ & Smooth & N/A & $A_{13}$ \\ [1ex] 
\hline 
\end{tabular}
\caption{Pairs $(X, D)$ constructed in subsections (\ref{pl_q})~--~(\ref{hyper}) and their associated $K3$ surfaces.} 
\label{table:pairs} 
\end{table}

\vspace{.60in}

\section{Construction of stable pairs via stable reduction}
\numberwithin{theorem}{section}
\counterwithin{theorem}{subsection}
\numberwithin{equation}{section}
\counterwithin{equation}{subsection}
In this section, we explain how to construct some of the pairs in the proof of Theorem~\ref{mainthm} via the Hassett-Keel program and stable reduction. We recall our discussion from Section $1$: In \cite{Mu14}, M\"uller shows that the final log canonical model of $\overline {\mathcal{M}}_6$ parametrizes quadric sections of $\Sigma_5$. In this section, we consider certain one-parameter degenerations over the germ of a smooth curve of quadric sections of $\Sigma_5$, where the generic fiber is smooth. We show that these families of pairs can be modified so that the new special fiber is a stable pair. In fact, we will recover some of the pairs containing special curves constructed in the previous section. The stable reduction process will involve applying the relative log minimal model program. We describe some examples below.

\vspace{.1in}

\subsection{Marked plane quintics of type (1,1,1,1,1)} \begin{proposition}
There exist quadric sections of $\Sigma_5$ with unique singularities of local analytic isomorphism type $y^5=x^5$.
\end{proposition}

\begin{proof}
Let $p_1, \dots, p_4$ denote points in $\PP^2$ in general position. Choose another point $p_5$, determining a smooth irreducible plane conic. Consider the union of this conic with the four lines connecting $p_5$ to each of the other $p_i$. We have constructed a reducible plane sextic curve with $5$ components meeting transversely at $p_5$. Blowing up $p_1, \dots, p_4$ and anti-canonically embedding the resulting surface in $\PP^5$ recovers $\Sigma_5$ with a quadric section of the desired singularity type.
\end{proof}

\begin{remark}
For future reference, let $C_0$ denote a curve in $\Sigma_5$ with this singularity type. Note that the log canonical threshold of the pair $(\Sigma_5, C_0)$ is $2/5<1/2$, hence this pair cannot be stable.
\end{remark}

\begin{proposition} \label{pl_q_mmp}
Let $(\mathcal{S}, \mathcal{C}) \rightarrow T$ be a family of surface-curve pairs over the germ of a smooth curve such that the generic fiber is a smooth quadric section of $\Sigma_5$ and the special fiber is $(\Sigma_5, C_0)$. There exists a family $(\mathcal{S'}, \mathcal{C'}) \rightarrow T'$ satisfying the following:
\begin{enumerate}

\item The generic fiber is isomorphic to the generic fiber of the original family. 
\item The special fiber is a stable pair with a unique $\frac{1}{4}(1,1)$ singularity and marked curve isomorphic to a smooth plane quintic.

\end{enumerate}
\end{proposition}

\begin{proof}
We first run local stable reduction for the singularity of $C_0$ in the special fiber. We view $(\mathcal{S}, \mathcal{C})\rightarrow T$ as a family of surfaces $\mathcal{S}$ containing $\mathcal{C}$. We perform a base change $t \mapsto t^5$, where $t$ is a uniformizing parameter of $T$. We denote this finite cover of $T$ by $T'$. We then blow up $\mathcal{S}$ at the singular point of $C_0$. 

This process yields a reducible surface $S_1 \cup S_2$ in the central fiber of the modified family. Let the double curve on $S_i$ be denoted by $B_i$. The surface $S_1$ is isomorphic to $\Sigma_4$ (a degree $4$ del Pezzo surface) marked with $C_1$ (the strict transform of $C_0$). A local computation shows that $S_2$ is isomorphic to $\PP^2$ marked with a smooth plane quintic $C_2$ meeting $B_2$ transversely. On $S_1$, the curve $B_1$ is the exceptional divisor when we blow up $\Sigma_5$ at the singular point of $C_0$, and on $S_2$, the curve $B_2$ is the hyperplane class.

Note that the special fiber of the resulting family is still not a stable pair. Consider the components of $C_1$, denoted $F_i$ for $i=1, \dots, 5$. If $H$ is the pullback of the hyperplane class from $\PP^2$ to $\Sigma_4$ and the $E_i$ are the exceptional divisors, we see explicitly:

\begin{enumerate}
\item $F_i=H-E_i-E_5$ for $i=1, \dots 4$.
\item $F_5=2H-\underset{j=1}{\overset{5}{\sum}}E_j$.
\end{enumerate}

These $F_i$ are all irreducible $(-1)$-curves and hence span extremal rays in the closure of the cone of effective curves $\overline {NE}(S_1)$. Consequently, these curves also span extremal rays in the closure of the relative cone of curves for our modified family.

The $F_i$ are all $K_{S_1}+\alpha C_1+B_1$-negative for all $\alpha>1/2$ by adjunction.  We explicitly construct flips of these curves. Note that after we flip one of these, each of the remaining $F_i$ can still be flipped via the same construction. A standard normal bundle computation shows that blowing up any one of the $F_i$ yields an exceptional divisor isomorphic to $\PP^1 \times \PP^1$, realizing the curve as one of the rulings. Projecting to the other ruling (this requires the contraction theorem) contracts $F_i$ on $S_1$ and blows up the point $B_2 \cap F_i$ on $S_2$.

Flipping all of the $F_i$ in this way yields a new surface $S'_1 \cup S'_2$, where $S'_1$ is isomorphic to $\PP^2$ and $S'_2$ is isomorphic to $\PP^2$ blown up at $5$ collinear points. Note that $S'_1$ has no marked curve and $S'_2$ is still marked with a curve isomorphic to a smooth plane quintic $C'_2$. The curve $B_1$ becomes a conic $B'_1$ in $S'_1$ after these flips. Hence the hyperplane class $H'$ in $S'_1$ is negative with respect to 
\begin{equation*}
K_{S'_1}+B'_1=-H',
\end{equation*}
which induces a divisorial contraction of $S'_1$. We are left with a surface $S''_2$, which is simply the contraction of the $(-4)$-curve $B'_2$ (the strict transform of $B_2$ after the flips) on $S'_2$. Hence $S''_2$ has a unique cyclic quotient singularity of type $\frac{1}{4}(1,1)$.
\end{proof}
\begin{remark}
We note that $S''_2$ is precisely the surface constructed in ~(\ref{pl_q}) corresponding to marked plane quintics of type $(1,1,1,1,1)$.
\end{remark}

\vspace{.2in}

\subsection{Marked trigonal curves of type $(2; [4])$}

\begin{proposition}
There exist quadric sections of $\Sigma_5$ with unique singularities of local analytic isomorphism type $y^3=x^7$.
\end{proposition}

\begin{proof}

By {\cite[1.10]{De90}}, there exists a plane sextic curve with such a singularity as well as four nodes in general position. Blowing up these nodes and anti-canonically embedding the resulting surface in $\PP^5$ recovers $\Sigma_5$ with a quadric section of the desired singularity type.
\end{proof}

\begin{remark}
For future reference, we will denote by $C_0$ a curve with this singularity type. The log canonical threshold of the pair $(\Sigma_5, C_0)$ is less than $1/2$, hence this pair cannot be stable. 
\end{remark}

\begin{proposition}
Let $(\mathcal{S}, \mathcal{C}) \rightarrow T$ be a family of surface-curve pairs over the germ of a smooth curve such that the generic fiber is a smooth quadric section of $\Sigma_5$ and the special fiber is $(\Sigma_5, C_0)$. There exists a family $(\mathcal{S'}, \mathcal{C'}) \rightarrow T'$ satisfying the following:
\begin{enumerate}

\item The generic fiber is isomorphic to the generic fiber of the original family. 
\item The special fiber is a stable pair with a unique $\frac{1}{20}(1,9)$ singularity and marked curve isomorphic to a smooth genus six trigonal curve. 

\end{enumerate}
\end{proposition}

\begin{proof}
Running local stable reduction for the family (see \cite{Has00}) yields a reducible surface $S=S_1 \cup S_2$ in the central fiber. Define $B_i$ as in Proposition~\ref{pl_q_mmp}. The surface $S_1$ is constructed by computing the embedded resolution of $C_0$ and contracting the exceptional divisors disjoint from its strict transform $C_1$. Let $F_i$, $i=1,2,3,4,5$, denote the exceptional divisors for this embedded resolution, where the indexing indicates the order in which these divisors appear as we repeatedly blow up points. In particular, the divisor $F_5$ denotes the exceptional curve which is not contracted post-embedded-resolution. We see that $S_1$ has two cyclic quotient singularities of type $\frac{1}{7}(1,4)$ and $\frac{1}{3}(1,2)$ along $B_1=F_5$. The surface $S_2$ is isomorphic to the weighted projective space $\PP(7,3,1)$, which has two cyclic quotient singularities of type $\frac{1}{7}(1,3)$ and $\frac{1}{3}(1,1)$ along $B_2$. Note that the singular points of $S_1$ are also singular on $S_2$. The curve $C_2 \subset S_2$ is smooth and trigonal of genus six avoiding the singularities and meeting $B_2$ transversely at one point.

Note that by adjunction applied to $C_1$ in $S_1$, the pair $(S_1 \cup S_2, C_1 \cup C_2)$ is not stable. The embedded resolution computation also reveals that $C_1$ is an irreducible $(-1)$-curve. Flipping $C_1$ as in Proposition~\ref{pl_q_mmp} amounts to contracting $C_1$ on $S_1$ while blowing up the point $C_2 \cap B_2$ on $S_2$. 

We denote the remaining reducible surface $S'_1 \cup S'_2$. We will show that the divisor
\begin{equation*}
-K_{S'_1}-F'_5
\end{equation*}
is ample, where $F'_5$ denotes the image of $F_5$ in $S'_1$ after flipping $C_1$. Let $\pi_1: S_1 \rightarrow \Sigma_5$ denote the sequence of blow-ups required for the embedded resolution of $C_0$, and let $\pi_2: S_1 \rightarrow S'_1$ denote the contraction of $F_i$ for $i=1,2,3,4$ and $C_1$. We note that $S'_1$ is $\mathbb{Q}$ factorial, so pulling back divisors makes sense. We compute
\begin{equation}
\pi_1^*(-2K_{\Sigma_5})=-2K_{S_1}+2F_1+4F_2+6F_3+12F_4+18F_5.
\end{equation}
On the other hand,
\begin{equation}
\pi_1^*(-2K_{\Sigma_5})=\pi^*(C_0)=C_1+3F_1+6F_2+7F_3+14F_4+21F_5,
\end{equation}
hence
\begin{equation} \label{5.3}
-2K_{S_1}=C_1+F_1+2F_2+F_3+2F_4+3F_5.
\end{equation}
Next, using ~(\ref{5.3}), we see that
\begin{equation} \label{5.4}
\pi_2^*(-2K_{S'_1}-2F'_5)=C_1+\frac{1}{7}F_1+\frac{2}{7}F_2+\frac{1}{3}F_3+\frac{2}{3}F_4+F_5.
\end{equation}

Now, let $C' \neq F'_5$ be a curve in $S'_1$. We want to show that the strict transform of $C'$ under $\pi_2$, henceforth denoted $\tilde C'$, is positive against the divisor in ~(\ref{5.4}). Suppose further that $\tilde C'$ is not any of the $F_i$ or $C_1$; then it is non-negative along each of these divisors. Suppose by contradiction that $\tilde C'$ is trivial along the $F_i$ and $C_1$. By construction, $\tilde C'$ is not $\pi_1$--exceptional, hence triviality against $C_1$ means that the scheme-theoretic intersection of the images of these two curves on $\Sigma_5$ is supported on the singular point of $C_0$. However, triviality of $\tilde C'$ against the $F_i$ implies that the image of $\tilde C'$ on $\Sigma_5$ is disjoint from $C_0$. This is absurd, since $C_0$ is an ample divisor. We have shown that $\tilde C'$ is indeed positive against the divisor in ~(\ref{5.4}).

It follows from this discussion that to verify ampleness of $-K_{S'_1}-2F'_5$, it is enough to check that the divisor in ~(\ref{5.4}) is positive against $F_5$, which it is. Hence we can divisorially contract $S'_1$ and we are left with a surface $S''_2$ with the desired cyclic quotient singularity. 

\end{proof}

\begin{remark}
Let $\phi_2: S'_2 \rightarrow S''_2$ denote the minimal resolution of $S''_2$. By explicitly blowing down $(-1)$-curves on $S'_2$, we obtain $\F_2$. Moreover, we see that $S''_2$ is precisely the surface constructed in the proof of Theorem~\ref{mainthm} associated to marked smooth trigonal curves of genus six and type $(2; [4])$. 

We also note that the divisorial contraction of $S'_1$ in this proof is of relative Picard number $5$, and the total space of the output family is not $\mathbb{Q}$-factorial. 
\end{remark}

\vspace{.2in}

\subsection{Marked trigonal curves of type $(0; 1,1,1,1)$}: Consider a triple plane conic $D$. Blow up four points on the curve in general position in $\PP^2$ to recover $\Sigma_5$ and consider the union of the strict transform $\tilde D$ with the exceptional divisors $E_i$. The resulting reducible curve $D'$ has class $-2K_{\Sigma_5}$. 

Consider a family of smooth quadric sections of $\Sigma_5$ degenerating to $D'$ as in the prior examples. Blow up the resulting family of surfaces along $\tilde D_{\text{red}}$. The exceptional divisor of this blow-up is isomorphic to $\PP^1 \times \PP^1$. So the new central fiber is a reducible surface $S_1 \cup S_2$, where $S_1 \cong \Sigma_5$ and $S_2 \cong \PP^1 \times \PP^1$. These surfaces are attached along one of the rulings of $\PP^1 \times \PP^1$. Each of the $E_i$ intersects the double curve at a single point. The strict transform of the blown up curve lies in $\PP^1 \times \PP^1$ and meets the double curve transversely in four points; these are precisely the intersection points of the $E_i$ with the double curve.

By adjunction applied to each $E_i$ in $S_1$, the reducible surface $S_1 \cup S_2$ and its marked curve do not form a stable pair. After flipping the $E_i$ as in Proposition~\ref{pl_q_mmp}, we obtain a reducible surface where one component is isomorphic to $\PP^2$ and the other component is isomorphic to $\PP^1 \times \PP^1$ blown up at four points along a ruling. We can divisorially contract the $\PP^2$ component as in Proposition~\ref{pl_q_mmp}. This amounts to contracting the $(-4)$-curve on this blow up of $\PP^1 \times \PP^1$, and we obtain the expected surface. 

\vspace{.1in}

\subsection{Bielliptic curves:} Consider a double plane cubic $D$. Blow up four points on the curve in general position in $\PP^2$ to recover $\Sigma_5$, and consider the strict transform $\tilde D$, which has class $-2K_{\Sigma_5}$. Consider a family of smooth quadric sections of $\Sigma_5$ degenerating to $\tilde D$ as in the prior examples. Blow up this family of surfaces along $\tilde D_{\text{red}}$. The exceptional divisor will be isomorphic to the minimal resolution of an elliptic cone of degree $5$. So the new central fiber consists of a reducible surface $S_1 \cup S_2$, where $S_1 \cong \Sigma_5$ and $S_2$ is isomorphic to the resolution of this cone. The strict transform of the blown up curve lies in the exceptional divisor, disjoint from the double curve.

Let the double curve on $S_i$ be denoted by $B_i$ as in Proposition~\ref{pl_q_mmp}.  Since $K_{S_1}+B_1$ is trivial, by taking the canonical model for the family, we can contract $S_1$. We obtain the expected elliptic cone of degree $5$. 

\begin{acknowledgments}
I am extremely grateful to my advisor, Maksym Fedorchuk, for introducing me to this project and for his guidance and patience throughout. I would also like to thank Brian Lehmann for very helpful discussions regarding the minimal model program. I would also like to express my gratitude to Changho Han for insightful conversations about the theory of stable pairs. 
\end{acknowledgments}

\thispagestyle{empty}
{\small
\markboth{References}{References}
\bibliographystyle{amsalpha} 
\bibliography{Refs}{}

\providecommand{\bysame}{\leavevmode\hbox to3em{\hrulefill}\thinspace}
\providecommand{\MR}{\relax\ifhmode\unskip\space\fi MR }
\providecommand{\MRhref}[2]{%
  \href{http://www.ams.org/mathscinet-getitem?mr=#1}{#2}
}
\providecommand{\href}[2]{#2}
\begin{thebibliography}{ACGH85}

\bibitem[Acc94]{Acc94}
Robert D.~M. Accola, \emph{Topics in the theory of {R}iemann surfaces}, Lecture
  Notes in Mathematics, vol. 1595, Springer-Verlag, Berlin, 1994. \MR{1329541}

\bibitem[ACGH85]{ACGH85}
E.~Arbarello, M.~Cornalba, P.~A. Griffiths, and J.~Harris, \emph{Geometry of
  algebraic curves. {V}ol. {I}}, Grundlehren der Mathematischen Wissenschaften
  [Fundamental Principles of Mathematical Sciences], vol. 267, Springer-Verlag,
  New York, 1985. \MR{770932}

\bibitem[AET19]{AET19}
Valery Alexeev, Philip Engel, and Alan Thompson, \emph{Stable pair
  compactification of moduli of k3 surfaces of degree 2}, arXiv preprint
  arXiv:1903.09742 (2019).

\bibitem[AK11]{AK11}
Michela Artebani and Shigeyuki Kond{\=o}, \emph{The moduli of curves of genus
  six and {$K3$} surfaces}, Transactions of the American Mathematical Society
  \textbf{363} (2011), no.~3, 1445--1462.

\bibitem[Art09]{Ar09}
Michela Artebani, \emph{A compactification of {$ M_3$} via {$K3$} surfaces},
  Nagoya Math. J. \textbf{196} (2009), 1--26. \MR{2591089}

\bibitem[BR95]{BR95}
Kurt Behnke and Oswald Riemenschneider, \emph{Quotient surface singularities
  and their deformations}, Singularity theory ({T}rieste, 1991), World Sci.
  Publ., River Edge, NJ, 1995, pp.~1--54. \MR{1378394}

\bibitem[CMJL12]{CMJL12}
Sebastian Casalaina-Martin, David Jensen, and Radu Laza, \emph{The geometry of
  the ball quotient model of the moduli space of genus four curves}, Compact
  moduli spaces and vector bundles, Contemp. Math., vol. 564, Amer. Math. Soc.,
  Providence, RI, 2012, pp.~107--136. \MR{2895186}

\bibitem[CML13]{CML13}
Sebastian Casalaina-Martin and Radu Laza, \emph{Simultaneous semi-stable
  reduction for curves with {ADE} singularities}, Trans. Amer. Math. Soc.
  \textbf{365} (2013), no.~5, 2271--2295. \MR{3020098}

\bibitem[Deg90]{De90}
A~I Degtyarev, \emph{Classification of surfaces of degree four having a
  nonsimple singular point}, Mathematics of the USSR-Izvestiya \textbf{35}
  (1990), no.~3, 607.

\bibitem[DH18]{DH18}
Anand Deopurkar and Changho Han, \emph{Stable log surfaces, admissible covers,
  and canonical curves of genus 4}, arXiv preprint arXiv:1807.08413 (2018).

\bibitem[Dol96]{Dol96}
I.~V. Dolgachev, \emph{Mirror symmetry for lattice polarized {$K3$} surfaces},
  J. Math. Sci. \textbf{81} (1996), no.~3, 2599--2630, Algebraic geometry, 4.
  \MR{1420220}

\bibitem[Hac01]{Hac01}
Paul Hacking, \emph{A compactification of the space of plane curves}, arXiv
  preprint math/0104193 (2001).

\bibitem[Hac04]{Hac04}
\bysame, \emph{Compact moduli of plane curves}, Duke Math. J. \textbf{124}
  (2004), no.~2, 213--257. \MR{2078368}

\bibitem[Har77]{Har77}
Robin Hartshorne, \emph{Algebraic geometry}, Springer-Verlag, New
  York-Heidelberg, 1977, Graduate Texts in Mathematics, No. 52. \MR{0463157}

\bibitem[Has99]{Has99}
Brendan Hassett, \emph{Stable log surfaces and limits of quartic plane curves},
  Manuscripta Math. \textbf{100} (1999), no.~4, 469--487. \MR{1734796}

\bibitem[Has00]{Has00}
\bysame, \emph{Local stable reduction of plane curve singularities}, J. Reine
  Angew. Math. \textbf{520} (2000), 169--194. \MR{1748273}

\bibitem[HL10]{HL10}
Donghoon Hyeon and Yongnam Lee, \emph{Log minimal model program for the moduli
  space of stable curves of genus three}, Math. Res. Lett. \textbf{17} (2010),
  no.~4, 625--636. \MR{2661168}

\bibitem[HP10]{HP10}
Paul Hacking and Yuri Prokhorov, \emph{Smoothable del {P}ezzo surfaces with
  quotient singularities}, Compos. Math. \textbf{146} (2010), no.~1, 169--192.
  \MR{2581246}

\bibitem[Kon00]{Kon00}
Shigeyuki Kondo, \emph{A complex hyperbolic structure for the moduli space of
  curves of genus three}, J. Reine Angew. Math. \textbf{525} (2000), 219--232.
  \MR{1780433}

\bibitem[Kon02]{Kon02}
\bysame, \emph{The moduli space of curves of genus 4 and {D}eligne-{M}ostow's
  complex reflection groups}, Algebraic geometry 2000, {A}zumino ({H}otaka),
  Adv. Stud. Pure Math., vol.~36, Math. Soc. Japan, Tokyo, 2002, pp.~383--400.
  \MR{1971521}

\bibitem[Kon05]{Ko05}
Kazuhiro Konno, \emph{Projected canonical curves and the {C}lifford index},
  Publ. Res. Inst. Math. Sci. \textbf{41} (2005), no.~2, 397--416. \MR{2138031}

\bibitem[KSB88]{KSB88}
J.~Koll\'{a}r and N.~I. Shepherd-Barron, \emph{Threefolds and deformations of
  surface singularities}, Invent. Math. \textbf{91} (1988), no.~2, 299--338.
  \MR{922803}

\bibitem[Laz09]{Laz09}
Radu Laza, \emph{Deformations of singularities and variation of {GIT}
  quotients}, Trans. Amer. Math. Soc. \textbf{361} (2009), no.~4, 2109--2161.
  \MR{2465831}

\bibitem[LO16]{LO16}
Radu Laza and Kieran O'Grady, \emph{G{I}{T} versus {B}aily-{B}orel
  compactification for quartic ${K3}$ surfaces}, arXiv preprint
  arXiv:1612.07432 (2016).

\bibitem[M{\"u}l14]{Mu14}
Fabian M{\"u}ller, \emph{The final log canonical model of {$\bar{M}$}6},
  Algebra \& Number Theory \textbf{8} (2014), no.~5, 1113--1126.

\bibitem[Pin74]{Pi74}
Henry~C. Pinkham, \emph{Deformations of algebraic varieties with
  {$G_{m}$}--action}, Soci\'{e}t\'{e} Math\'{e}matique de France, Paris, 1974,
  Ast\'{e}risque, No. 20. \MR{0376672}

\bibitem[Pro17]{Pr17}
Yuri Prokhorov, \emph{Log canonical degenerations of del {P}ezzo surfaces in
  {Q}-{G}orenstein families}, arXiv preprint arXiv:1703.10229 (2017).

\bibitem[Sha79]{Sh79}
Jayant Shah, \emph{Insignificant limit singularities of surfaces and their
  mixed {H}odge structure}, Ann. of Math. (2) \textbf{109} (1979), no.~3,
  497--536. \MR{534760}

\bibitem[Sha80]{Sh80}
\bysame, \emph{A complete moduli space for {$K3$} surfaces of degree 2}, Annals
  of Mathematics \textbf{112} (1980), no.~3, 485--510.

\bibitem[Yan96]{Ya96}
Jin-Gen Yang, \emph{Sextic curves with simple singularities}, Tohoku Math. J.
  (2) \textbf{48} (1996), no.~2, 203--227. \MR{1387816}

\end{thebibliography}

}

	\bigskip

\end{document}